\documentclass[11pt,reqno]{amsart}
\usepackage{mathrsfs}
\usepackage{}

\def\subjclass#1{{\renewcommand{\thefootnote}{}%
\footnote{\emph{Mathematics Subject Classification (2010):} #1}}}

\usepackage{amsfonts}
\usepackage{amssymb}
\DeclareMathOperator{\curl}{curl} \DeclareMathOperator{\dv}{div}
\DeclareMathOperator{\supp}{supp}

\date{\today}

\hoffset=- 2cm \voffset=- 0cm 
 \theoremstyle{plain}
\newtheorem{Thm}{Theorem}

\newtheorem{Rem}[Thm]{Remark}
\newtheorem{Lem}[Thm]{Lemma}

\usepackage{amssymb}
\usepackage{color}

\newcommand {\p}{\partial}
\newcommand{\q}{\quad}
\newcommand{\qq}{\qquad}

\def\O{\Omega}

\def\E{\mathbf E}

\def\e{\mathbf e}
\def\n{\mathbf n}
\def\r{\mathbf r}

\def\u{\mathbf u}
\def\w{\mathbf w}

\def\v{\vskip}

\errorcontextlines=0 \numberwithin{equation}{section}
\numberwithin{Thm}{section}

\begin{document}
\large

\title[Hardy-type inequalities for vector fields in bounded domains]
{Hardy-type inequalities for vector fields
with the tangential components vanishing}

\author[]{Xingfei Xiang}
\author[]{Zhibing Zhang}

\address{Xingfei Xiang: Department of Mathematics,
Tongji University, Shanghai 200092, P.R. China; }
\email{xiangxingfei@126.com}

\address{Zhibing Zhang: Department of Mathematics,
East China Normal University, Shanghai 200241, P.R. China; }
\email{zhibingzhang29@126.com}

\thanks{ }



\begin{abstract}
This note studies the Hardy-type inequalities
 for vector fields with the $L^1$ norm of the $\curl$.
In contrast to the well-known results in the
whole space for the divergence-free vectors,
we generalize the Hardy-type inequalities
to the bounded domains and to the
non-divergence-free vector fields with the tangential
components on the boundary vanishing.

\end{abstract}
\maketitle

\section{Introduction}

This note is devoted to establish the Hardy-type
inequality for vector fields in $L^1$ space
in 3-dimensional bounded domains.
 We prove that for the vector field $\u$
 with the tangential
components on the boundary vanishing,
 the $L^1$ norm of $\u/|x|$ can be controlled
 by the $L^1$ norm of $(1+\ln |x|)\dv\u$
 and the $L^1$ norm of $\curl\u$.

This work belongs to the field of the $L^1$ estimate
for vector fields.
Starting with the pioneering work
by Bourgain and Brezis in \cite{BB1},
the $L^1$ estimate has been well studied by
many mathematicians, see [2-7, 10-12, 15, 16, 18]
and the references therein.
In particular, Bourgain and Brezis in \cite{BB}
obtained the delicate $L^{3/2}$ estimate
for the divergence-free vectors on
 the torus $\mathbb{T}^3$.
Maz'ya in \cite{MA} (also see Bousquet
and Van Schaftingen's more general result in \cite{BS} by introducing
the cancellation condition)
obtained a Hardy-type inequality
for the divergence-free vector
fields $\u$ (not direct but implied)
\begin{equation}\label{1.2}
\left\|\frac{\mathbf u}{|x|}\right\|_{L^1(\mathbb{R}^3)}
\leq C\|\curl\mathbf \u\|_{L^1(\mathbb{R}^3)}.
\end{equation}
This actually gives the essential answer to
the problem raised by Bourgain and Brezis
in \cite[open problem 1]{BB}. Bousquet and Mironescu
give an elementary proof of \eqref{1.2} in \cite{BM}.

In this note we consider the problem in bounded domains,
in particular with the singularity
(the origin) being on the boundary.
For the case of the singularity being in the
interior of the domain,  we can easily get the
similar estimate \eqref{1.2} in bounded domains
by taking the cut-off method.
However, if the singularity is on the boundary,
the usual flattening boundary and the
localization by partition of unity does not work.
The main reason is that by taking the
flattening boundary there would arise the $L^1$
norm of $\nabla \u$, this term can't be
controlled by the $L^1$ norm of $\curl\u$
for the vector $\u$ being divergence-free, and hence
a new approach should be considered.

To get around this difficulty we apply  the
Helmholtz-Weyl decomposition for vector fields
 in bounded domains
(see \cite[Theorem 2.1]{KY}):
$$\u=\curl\mathbf w_{\u}+\nabla p_{\u}+\mathscr{H}_{\u},
$$
where $\mathscr{H}_{\u}$ is the harmonic part depending only
on the domain. Our strategy is to get the estimates
for the curl part $\curl\mathbf w_{\u}$
and the gradient part $\nabla p_{\u}$ respectively.
Thanks to Solonnikov's work in \cite{SO1,SO2} (also see
Beir\~{a}o da Veiga and Berselli's work in \cite[p.606]{BVB}),
the vector $\mathbf w_{\u}$ in the curl part satisfies
Petrovsky type elliptic system, and hence there exists a single
Green's matrix $\mathcal {G}(x,y)$ such that
$$\w(x)=\int_{\O} \mathcal {G}(x,y) \curl\u(y) dy.
$$
Based on the estimate for the Green's matrix,
we can obtain the estimate on the curl part.
For the gradient part $\nabla p_{\u}$ we can use the
classical elliptic theory to get
the estimate.

Before stating the main result, we make the
following assumption on the domain:

 (O)\, $\O$ is a bounded in
$\mathbb{R}^3$ with smooth boundary,
in all cases considered here
the class $C^2$ will be sufficient.
The second Betti number is $0$ which
is understood as there is no
holes in the domain.

Denote by $\nu(x)$ the unit outer normal vector at
$x\in \p\O.$ The main result now reads:

\begin{Thm}\label{thm1.1}
Assume that the domain $\O$ satisfies the assumption $(O)$.
Then for any $\u\in C^{1}(\bar{\O}, \mathbb{R}^3)$
with $\nu\times\u=0$ on $\p\O,$  we have
\begin{equation}\label{1.6}
\left\| \frac{\u}{|x|}\right\|_{L^{1}(\O)}
\leq C\left(\left\|\ln |x|\dv\u\right\|_{L^1(\O)}
+\|\curl\u\|_{L^{1}(\O)}\right),
\end{equation}
where the constant $C$ depends only on the domain $\O$.
\end{Thm}

\begin{Rem} We need to mention that
\begin{itemize}
\item[(i)] The proof method of this theorem is not applicable for the vector fields
with the normal components on the boundary vanishing. The reason is that
the key step we used is the zero extension of $\curl\u$
outside of the domain, but this does not hold for the vector fields with the normal
components on the boundary vanishing.

\vspace {0.2cm}

\item[(ii)] By a similar discussion, one can get the estimate
for elliptic system associating with the Hardy-type inequality.
Let $\mathbf f\in C^{1}(\bar{\O},\mathbb R^3)$ with
 $\dv\mathbf f=0$ in $\O$ and $\nu\cdot\mathbf f=0$ on
the boundary. Then for the elliptic system $\mathscr{L} \u=\mathbf f$
with the form of the solution can be expressed by
$$ \u=\int_{\O} \mathcal {G}(x,y) \mathbf f(y) dy,
$$
where the Green's matrix $\mathcal {G}(x,y)$ satisfy the inequality \eqref{6.1},
 we have
$$\left\| \frac{\u}{|x|^2}\right\|_{L^{1}(\O)}
\leq C \left\|\mathbf f\right\|_{L^1(\O)}.
$$
However, this may not be true for the single elliptic equation. The typical example is
$\mathscr{L}=\Delta$ with zero boundary condition.
\end{itemize}
\end{Rem}

The organization of this paper is as follows. In Section 2,
we will give the  proof of Theorem \ref{thm1.1}.
In Section 3, the Hardy-type inequality in $L^p$
space with $1<p<3/2$ will be treated.
We will show that the $L^p$ norm of $\u/|x|$
can be controlled by the $L^p$ norm of the
$\dv\u$ and the $\curl\u$ whether or not the singularity
is on the boundary.

Throughout the paper, the bold typeface is used to
 indicate vector quantities; normal typeface will
 be used for vector components and for scalars.
\v0.1in

\section{Hardy-type inequalities for vector fields with $L^1$ data}\label{section2}

The key step in the proof of the main theorem
is the estimate for the singular integral involving
the operator curl. This estimate was first obtained
by Maz'ya in \cite{MA} in the case of the kernel
being the Newtonian potential and of the domain being
the entire space. The case where the kernel being
the Green's function associating with the elliptic
operator in the entire space
 was considered by Bousquet and Van Schaftingen
 (see \cite[Lemma 2.2]{BS}).
We generalized their kernel to a more general case.
The method of our proof  goes back to work
by Bousquet and Mironescu in \cite{BM}.

\begin{Lem}\label{lem4.1}
Suppose that the function
$A(x,y)\in C^1(\O\times\mathbb R^3)$
for $x\neq y$
satisfying
\begin{equation}\label{6.1}
(i)\q |A(x,y)|\leq \frac{C}{|x-y|^2};\q (ii)\q \left|\nabla_y A(x,y)\right|\leq \frac{C}{|x-y|^3}.
 \end{equation}
Let $\mathbf\Psi\in L^1(\mathbb R^3, \mathbb R^3)$ with $\dv\mathbf\Psi=0.$ Then there exists a constant $C$ such that
\begin{equation}\label{4.1}
\left\| \frac{1}{|x|}\int_{\mathbb{R}^3}A(x,y) \mathbf\Psi (y)dy\right\|_{L^{1}(\O)}
\leq C\left\|\mathbf\Psi\right\|_{L^1(\mathbb{R}^3)}.
\end{equation}
\end{Lem}

\begin{proof}
For simplicity of the notations, we let
$$
I_{\Sigma}:=\int_{\Sigma}A(x,y)\mathbf\Psi(y) dy.
$$
Then write
\begin{equation}\label{4.3}
\int_{\mathbb R^3}
 A(x,y) \mathbf\Psi(y) dy\\
 =I_{{\{|y|>2|x|\}}}+I_{\left\{\frac{|x|}{4}<|y|<2|x|\right\}}
 +I_{\left\{|y|<\frac{|x|}{4}\right\}}.
\end{equation}

The estimation of these integrals are achieved as follows.
For all $x,y$ satisfying $2|x|<|y|,$
the inequality $|x-y|\geq |y|/2$ holds, then using Fubini's theorem, we obtain
\begin{equation}\label{4.5}
\aligned
\int_{\mathbb R^3}\frac{1}{|x|}\left| I_{{\{|y|>2|x|\}}}\right|dx
&\leq \int_{\mathbb R^3}|\mathbf\Psi(y)|
\int_{\left\{|x|<\frac{|y|}{2}\right\}\bigcap\O}|A(x,y)|\frac{1}{|x|} dxdy\\
&\leq C\int_{\mathbb R^3}|\mathbf\Psi(y)|\frac{1}{|y|^2}
\int_{\left\{|x|<\frac{|y|}{2}\right\}}\frac{1}{|x|} dxdy.
\endaligned
\end{equation}
It is easy to see that the last term of the above inequality
can be controlled by the $L^1$ norm of $\mathbf\Psi(x).$
We now estimate the second term in the right side of \eqref{4.3}.
 Using Fubini's theorem again, it follows that
\begin{equation}\label{4.2}
\aligned
\int_{\mathbb R^3}\frac{1}{|x|}\left| I_{\frac{|x|}{4}<|y|<2|x|}\right|dx
&\leq \int_{\mathbb R^3}|\mathbf\Psi(y)|
\int_{\left\{\frac{|y|}{2}<|x|<4|y|\right\}}\frac{1}{|x||x-y|^2} dxdy\\
&\leq C\int_{\mathbb R^3}|\mathbf\Psi(y)|
\int_{\left\{|x-y|<5|y|\right\}}\frac{1}{|y||x-y|^2} dxdy,
\endaligned
\end{equation}
the last term in the above inequality can also be controlled
by $\|\mathbf\Psi(x)\|_{L^1(\mathbb R^3)}.$

We now estimate the integral involving the term $I_{\{|y|<|x|/4\}}.$
Take the cut-off function $\eta(t)$ such that $\eta(t)=0$ for $t>1/2,$ $\eta(t)=1$
for $0<t<1/4,$ and $|\eta'(t)|\leq 8.$ Applying the equality
$$\int_{\mathbb R^3} \dv(y_i A(x,y)\eta(|y|/|x|)\mathbf\Psi(y) )dy
=0\q\q \text{for } i=1,2,3,
$$
and then using $\dv\mathbf\Psi=0$ we have
$$
\aligned
  &\int_{\mathbb R^3}  A(x,y)\eta(|y|/|x|)\mathbf\Psi(y) dy\\
=&
-\int_{\mathbb R^3}  \nabla_y(A(x,y)\eta(|y|/|x|))
\cdot\mathbf\Psi(y)(y_1,y_2,y_3) dy.
\endaligned
$$
The conditions (i) and (ii), for $|y|<|x|/2,$ imply that
$$
|\nabla_y(A(x,y)\eta(|y|/|x|))|
\leq C\left(\frac{1}{|x-y|^3}+\frac{1}{|x||x-y|^2}\right)
\leq C\frac{1}{|x|^3}.
$$
This shows that
\begin{equation}\label{4.4}
\int_{\O}\frac{1}{|x|}\left|\int_{\mathbb R^3}
A(x,y)\eta(|y|/|x|)\mathbf\Psi(y) dy \right|dx
\leq C \int_{\mathbb R^3}|\mathbf\Psi(y)|
\int_{\left\{|x|>2|y|\right\}}\frac{|y|}{|x|^4} dxdy.
\end{equation}
The last term in the above inequality can be controlled
by $\|\mathbf\Psi(x)\|_{L^1(\mathbb R^3)}.$
Noting that
$$
\aligned
&\int_{\mathbb R^3}  A(x,y)\eta(|y|/|x|)\mathbf\Psi(y) dy\\
=&I_{\left\{|y|<\frac{|x|}{4}\right\}}
+\int_{\left\{\frac{|x|}{4}<|y|<\frac{|x|}{2}\right\}}
 A(x,y)\eta(|y|/|x|)\mathbf\Psi(y) dy.
\endaligned
$$
The estimation involving the last term of the above equality
can be obtained by \eqref{4.2}. Thus from \eqref{4.4} it follows that
$$
\int_{\O}\frac{1}{|x|}\left| I_{\left\{|y|<\frac{|x|}{4}\right\}}\right|dx
\leq C\|\mathbf\Psi(x)\|_{L^1(\mathbb R^3)}.
$$
Plugging \eqref{4.5}, \eqref{4.2} and the above inequality
back to \eqref{4.3} we obtain \eqref{4.1}. We finish our proof.
\end{proof}

\begin{Rem} It is easy to see that this lemma is still true
if replaced the scalar function
$A(x,y)$ by a matrix $\mathcal {G}(x,y)$
satisfying the inequalities in \eqref{6.1}.
\end{Rem}

By a similar discussion of Lemma \ref{lem4.1}
and applying Fubini's theorem, we can get the
estimate of the singular integral for scalar functions.

\begin{Lem}\label{lemma2.2}
Suppose that the function $\ln |x|\Psi\in L^1(\O)$
and assume that the function
$A(x,y)\in C^1(\bar{\O}\times\bar{\O})$
for $x\neq y$ satisfying the inequalities in \eqref{6.1}.
Then there exists a constant $C$ depending only on the domain
such that
\begin{equation}\label{2.1}
\left\| \frac{1}{|x|}\int_{\O}A(x,y)
\Psi(y) dy\right\|_{L^{1}(\O)}
\leq C\left\|(1+\ln |x|)\Psi\right\|_{L^1(\O)}.
\end{equation}
\end{Lem}
For the vector $\mathbf\Psi=\curl \mathbf\Phi$ with $\nu\times\mathbf\Phi=0$ on the boundary,
the estimate in Lemma \ref{lemma2.2} can be improved.
\begin{Lem}\label{lem2.2}
Let $\mathbf\Phi\in C^1(\bar{\O}, \mathbb R^3)$
with $\nu\times\mathbf\Phi=0$ on the boundary. Suppose that the function
$A(x,y)\in C^1(\bar{\O}\times\bar{\O})$
for $x\neq y$
satisfying the inequalities in \eqref{6.1}, then
there exists a constant $C$ such that
\begin{equation}\label{2.2}
\left\| \frac{1}{|x|}\int_{\O}
A(x,y)\curl \mathbf\Phi dy\right\|_{L^{1}(\O)}
\leq C\left\|\curl \mathbf\Phi\right\|_{L^1(\O)}.
\end{equation}
\end{Lem}
\begin{proof}
Let $\tilde{\mathbf\Phi}$ be the zero extension
 of the vector $\mathbf\Phi$ outside of the $\O.$ Then
 $\curl \tilde{\mathbf\Phi}=0$ in $\mathbb R^3$
 in the sense of distribution, and we get
$$
\left\| \frac{1}{|x|}\int_{\O}
A(x,y)\curl \mathbf\Phi dy\right\|_{L^{1}(\O)}
= \left\| \frac{1}{|x|}\int_{\mathbb R^3}
A(x,y)\curl \tilde{\mathbf\Phi} dy\right\|_{L^{1}(\O)}.
$$
From Lemma \ref{lem4.1}, it follows that
$$
\left\| \frac{1}{|x|}\int_{\O}
A(x,y)\curl \mathbf\Phi dy\right\|_{L^{1}(\O)}
\leq C \left\|\curl \tilde{\mathbf\Phi}\right\|_{L^1(\mathbb R^3)}
\leq C \left\|\curl \mathbf\Phi\right\|_{L^1(\O)}.
$$

\end{proof}

We now give the proof of the main theorem.

\begin{proof}[Proof of Theorem \ref{thm1.1}]

From the Helmholtz-Weyl decomposition (see \cite[Theorem 2.1]{KY}),
for every $\u\in C^1(\bar{\O})$ there
 exists a decomposition
\begin{equation}\label{4.12}
 \u=\nabla p_{\u}+\curl \w_{\u},
\end{equation}
where the function $p_{\u}\in W^{2,p}(\O) $ satisfying
\begin{equation}\label{4.13}
\Delta p_{\u}=\dv\u\q\text{ \rm in } \O,
\quad\qq p_{\u}=0\q\text{ \rm on }\p\O;
\end{equation}
the vector $\w_{\u}\in X^{2,p}_{n}(\O)$ with $X^{2,p}_{n}$ defined by
$$
 X^{2,p}_{n}(\O)\equiv\left\{\w\in
W^{2,p}(\O)~:~\dv\,\w=0,~ \nu\cdot\w=0 \text{ on }\p\O \right\}
$$
and the vector $\w_{\u}$ satisfying the elliptic system
\begin{equation}\label{4.14}
\begin {cases}
\curl\curl \w_u=\curl\u & \text{ \rm in }\O,\\
\dv\w_u=0 &\text{ \rm in }  \O,\\
\nu\times\curl\w_{\u}=\nu\times\u=0 &\text{ \rm on }  \p\O,\\
\nu\cdot\w_{\u}=0 &\text{ \rm on }  \p\O.
\end {cases}
\end{equation}

By the classical elliptic equation theory, we see that
the solution $p_{\u}$ of the equation \eqref{4.13} has the form
$$ p_{\u}=\int_{\O} G_1(x,y) \dv\u(y) d(y)
$$
and Green's function $G_1(x,y)$ satisfies the
inequality \eqref{6.1}. Lemma \ref{lemma2.2} gives
\begin{equation}\label{4.17}
\left\|\frac{\nabla p_{\u}}{|x|}\right\|_{L^1(\O)}
\leq C \left\|(1+\ln|x|)\dv\u\right\|_{L^1(\O)}.
\end{equation}

We now estimate the term involving the operator curl. Note that the elliptic system \eqref{4.14}
is of Petrovsky type (see Solonnikov \cite{SO1,SO2}, also see
the reference \cite[p.606]{BVB} by Beir\~{a}o da Veiga and Berselli). Therefore,
the solution of \eqref{4.14} can be expressed by
$$\w(x)=\int_{\O} \mathcal {G}(x,y) \curl\u(y) dy.
$$
where the Green's matrix $\mathcal {G}(x,y)$ is written as
$$
\mathcal {G}(x,y)=\mathbf G_2(x,y)+\mathbf R(x,y).
$$
The leading term $\mathbf G_2(x,y)$ satisfies the estimate (see \cite[p.608]{BVB})
\begin{equation}\label{4.15}
 \left|D_x^{\alpha} D_y^{\beta}\mathbf G_2(x,y)\right|
\leq \frac{C(\alpha,\beta,\O)}{|x-y|^{\alpha+\beta+1}}
\end{equation}
and the matrix $\mathbf R(x,y)$ satisfies (see \cite[p.610]{BVB})
\begin{equation}\label{4.16}
\left|D_x^{\alpha} D_y^{\beta}\mathbf R(x,y)\right|
\leq \frac{C(\alpha,\beta,\O)}{|x-y|^{\alpha+\beta+1+\gamma}}\q\q \text{with } \gamma>0.
\end{equation}
Note that
$$
\curl\w(x)=\int_{\O} \curl_x (G^1,G^2, G^3)(x,y)\curl\u(y) dy
$$
where $(G^1,G^2, G^3)=\mathcal {G}(x,y)^{T}$ is the
row vector of the matrix $\mathcal {G}(x,y)$.
Then by the estimates \eqref{4.15} and \eqref{4.16}, Lemma \ref{2.2}
is applicable for $\curl\w(x),$ and we have
\begin{equation}\label{4.18}
\left\|\frac{\curl\w_{\u}}{|x|}\right\|_{L^1(\O)}
\leq C \left\|\curl\u\right\|_{L^1(\O)}.
\end{equation}
Combing the estimates \eqref{4.17} and \eqref{4.18}, we complete the proof.
\end{proof}

\section{Hardy-type inequalities for vector fields with $L^p$ data}

To show the Hardy-type inequalities for vector fields
with $L^p$ data, we first give
the estimate on the vector field itself by the $L^1$ norm
of the operators div and curl.

\begin{Lem}\label{lem2.3}
Assume that the domain $\O$ satisfies the assumption $(O),$
and let  $\u\in C^{1}(\bar{\O})$ with $\nu\times\u=0$ on $\p\O$.
 Then for any $1\leq p<3/2$ we have
\begin{equation}\label{2.3}
\left\| \u\right\|_{L^{p}(\O)}
\leq C(p, \O)\left(\|\dv\u\|_{L^1(\O)}+\|\curl\u\|_{L^{1}(\O)}\right),
\end{equation}
where  the  constant $C$ depends only on $p$ and the domain $\O.$
\end{Lem}
\begin{proof}
From the fundamental theorem of vector calculus, it follows that
\begin{equation}\label{2.4}
4\pi\u=-\nabla\int_{\O}\frac{1}{|x-y|}\dv\u(y)
dy+\curl\int_{\O}\frac{1}{|x-y|}\curl\u(y)
dy+\nabla v,
\end{equation}
where the function $v$ is defined by
$$v(x)=\int_{\p\O}\frac{1}{|x-y|}\nu\cdot\u(y)\mathrm{d}S_y.
$$
We estimate each term in \eqref{2.4}.
The $L^p$ estimates on the first and the second terms
in the right side of \eqref{2.4} can be obtained
by applying the Minkowski's integral inequality.
Therefore, it is necessary to show the  estimate on
$v(x)$.
By Lemma A.1 in Appendix in \cite{Xiang}
and by Claim 1 in the proof
of Theorem 1.1 in \cite{Xiang}, for $1<p<3/2$ we have
\begin{equation}\label{2.5}
\left\| \nabla v(x)
\right\|_{L^{p}(\O)}\leq C(p,\O)\|\nu\cdot\u\|_{W^{-1/p,p}(\p\O)}
\leq C(p,\O)\|J\|_{W^{-1/p,p}(\p\O)},
\end{equation}
where $J$ is defined by
$$J=2\pi\left(\nu,\mathrm{grad}\int_{\O}\frac{1}{r}\dv\u
\mathrm{d}y-\curl\int_{\O}\frac{1}{r}\curl\u \mathrm{d}y \right).
$$

We now estimate $J,$ and let
$$w=\int_{\O}\frac{1}{|x-y|}\dv\u
\mathrm{d}y.
$$
 Green's formula and the trace theorem in $W^{1,q}(\O)$ yield that
\begin{equation}\label{2.6}
\left\|\frac{\p w}{\p\nu}\right\|_{W^{-1/p,p}(\p\O)}\leq
C(p,\O)\left(\left\|\nabla w\right\|_{L^{p}(\O)}
+\|\dv\u\|_{(W^{1,q}(\O))^{*}}\right),
\end{equation}
where $q$ is the exponent conjugate to $p$,
$(\cdot)^{*}$ denotes the dual space of $(\cdot).$
Using the Minkowski's integral inequality again, we get
\begin{equation}\label{2.7}
\left\|\nabla w\right\|_{L^{p}(\O)}
\leq C(p,\O)\|\dv\u\|_{L^1(\O)}.
\end{equation}
Since $W^{1,r}$  is continuously embedded
into $L^{\infty}$ for  $r>3,$ $L^1$ is
continuously embedded into  $(W^{1,r}(\O))^{*}.$
Hence, for any  $1<p<3/2$, \eqref{2.6}
and \eqref{2.7} yield that
\begin{equation}\label{2.8}
\left\|\frac{\p }{\p\nu}\int_{\O}\frac{1}{|x-y|}
\dv\u\right\|_{W^{-1/p,p}(\p\O)}
\leq C\|\dv\u\|_{L^1(\O)}.
\end{equation}
The trace theorem (\cite[Lemma 1]{FM})
in the space $\{\u\in L^p(\O), \dv\u\in L^p(\O)\}$ gives
$$
\left\|\left(\nu,\curl\int_{\O}\frac{1}{r}\curl\u
\mathrm{d}y\right)\right\|_{W^{-1/p,p}(\p\O)} \leq
C(p,\O)\left\|\curl\int_{\O}\frac{1}{r}\curl\u
\mathrm{d}y\right\|_{L^{p}(\O)}.
$$
Thus the Minkowski's integral inequality, for $1<p<3/2$, implies that
\begin{equation}\label{2.9}
\left\|\left(\nu,\curl\int_{\O}\frac{1}{r}\curl\u
\mathrm{d}y\right)\right\|_{W^{-1/p,p}(\p\O)}
\leq  C(p,\O)\|\curl\u\|_{L^{1}(\O)}.
\end{equation}
 From \eqref{2.8} and \eqref{2.9}, we obtain the estimate on $J$,
\begin{equation}\label{2.10}
\|J\|_{W^{-1/p,p}(\p\O)}\leq C(p,\O)\left(\|\dv\u\|_{L^1(\O)}
+\|\curl\u\|_{L^{1}(\O)}\right).
\end{equation}
Thus  from \eqref{2.5} and \eqref{2.10} we obtain \eqref{2.3}
if $1<p<3/2$. For $p=1$ we apply H\"{o}lder's inequality in \eqref{2.3}, and
hence we finish our proof.
\end{proof}

We now show the Hardy-type inequalities with $L^p$ data for the vectors.
\begin{Thm}\label{thm3.2}
Assume that the domain $\O$ satisfies
the assumption $(O)$ with the origin ${\bf 0}\in\bar{\O}$.
Let $\u\in W^{1,p}(\O, \mathbb R^3)$ with $1<p<3/2$ and $\nu\times\u=0$
on $\p\O$ in the sense of trace.
Then there exists a positive constant
$C$ depending only on $p$
 and the domain $\O$ such that
\begin{equation}\label{6.22}
\left\| \frac{\u}{|x|}\right\|_{L^{p}(\O)}
\leq C(p, \O)\left(\left\|\dv\u\right\|_{L^p(\O)}
+\|\curl \u\|_{L^p(\O)}\right).
\end{equation}
\end{Thm}
\begin{proof}
Since the domain is in $C^{2}$ class,
there exist a positive constant $\epsilon$
being sufficiently small and a
$C^{2}$ diffeomorphism that
straightens the boundary in the neighborhood
       $$\mathcal {N}:=\left\{x\in\O:\mathrm{dist}(x,\p\O)
       \leq \epsilon\right\}
       $$
of the boundary. The number $\epsilon$ depends on the domain.

Case (i): $d:=\mathrm{dist}(\mathbf{0},\p\O)>0.$
Decomposing the integral and then by Minkowski's inequality, we get
\begin{equation}\label{2.23}
\left\|\frac{\u}{|x|}\right\|_{L^{p}(\O)}\leq
\left\| \frac{\eta_1\u}{|x|}\right\|_{L^{p}(\O)}
    +\left\| \frac{(1-\eta_1)\u}{|x|}\right\|_{L^{p}(\O)},
\end{equation}
where the function $\eta_1$ is defined by
\begin{equation}\label{5.12}
0\leq \eta_1(x)\leq 1,\q \eta_1(x)=1 \q \text{in } B_{d/2},
\q \supp\eta_1 \subset B_{d},
\q |D\eta_1|\leq \frac{4}{d}.
\end{equation}
For the first term in the right side of \eqref{2.23}, applying
the classical Hardy inequality and using the $W^{1,p}$ estimate
for vector fields with the tangential components on the boundary
vanishing, we have
$$\left\| \frac{\eta_1\u}{|x|}\right\|_{L^{p}(\O)}
\leq C(p,\O) \left\| \nabla(\eta_1\u)\right\|_{L^{p}(\O)}
\leq C(p, \O)\left(\left\|\dv\u\right\|_{L^p(\O)}
+\|\curl \u\|_{L^p(\O)}\right).
$$
The estimate on the last term of \eqref{2.23} is easy to obtain
since $(1-\eta_1)\u=0$ in the neighborhood of the origin, and thus
$$\left\| \frac{(1-\eta_1)\u}{|x|}\right\|_{L^{p}(\O)}
\leq C(p,\O) \left\| \u\right\|_{L^{p}(\O)}
\leq C(p, \O)\left(\left\|\dv\u\right\|_{L^p(\O)}
+\|\curl \u\|_{L^p(\O)}\right).
$$
In this case, the inequality \eqref{6.22} follows immediately.

Case (ii): $\mathbf{0}\in\p\O.$ We now consider the problem in the
neighborhood $\mathcal {N}$ of $\mathbf{0}$ and
 take the grid of the curvature lines as the
curvilinear coordinate system.
Note that the forms of  the divergence and the $\curl$
are invariant under orthogonal transformations. Therefore,
without loss of generality,  we assume the unit
inward normal vector of $\p\O$ at the
point $\mathbf{0}$ is $k=(0,0,1).$
 We introduce new variables $y_1$ and
$y_2$ such that $\r(y_1,y_2)$ represents  the portion of $\p\O$ near
$\mathbf{0}$ with $\r(0,0)=\mathbf{0}$ and the $y_1-$ and $y_2-$curves on $\p\O$
are the lines of principle curvatures. By the rotations, we assume one of the
principle direction at $\mathbf{0}$ is $\e_1=i=(1,0,0),$ the
other principle direction is $\e_2=j=(0,1,0).$
We take the diffeomorphism map $\mathcal {F}$ near boundary as
follows:
\begin{equation}\label{6.12}
x=\mathcal {F}(y)=\r(y_1,y_2)+y_3\n(y_1,y_2),
\end{equation}
where $\n(y_1,y_2)$ is the unit inner normal vector at the point
$\r(y_1,y_2)\in \p\O.$ Let
$$
G_{ij}(y)=\p_i\mathcal
{F}(y)\cdot\p_j\mathcal {F}(y),\q G(y)=G_{11}(y)G_{22}(y).
$$
Then we get
\begin{equation}\label{6.13}
G_{ii}(y)=1+o(|y|).
\end{equation}
Let $\hat{\u}(y)$ be the representation of the vector
$\u(x)$ under the new coordinate framework $\left\{\E_1,
\E_2, \E_3\right\}$, that is
$$\hat{\u}(y)=\hat{u}_1(y)\E_1+\hat{u}_2(y)\E_2+\hat{u}_3(y)\E_3=\u(x).
$$

Let $\O_{t}=\O\bigcap B_{t}(X_0)$ with $B_r=B_r(\bf 0)$ the ball
 center $\bf 0$ and radius $r$, and assume
that $\mathcal {F}$  defined by \eqref{6.12}
maps some domain
$\tilde{\O}_{\epsilon}\subset \mathbb R^3_+$ onto a subdomain
$\O_{\epsilon},$ maps the domain
$\tilde{\O}_{\epsilon/2}\subset \mathbb R^3_+$ onto a subdomain
$\O_{\epsilon/2}$.
We decompose the integral by two parts:
\begin{equation}\label{6.25}
\left\|\frac{\u}{|x|}\right\|_{L^{p}(\O)}\leq
\left\| \frac{\eta_2\u}{|x|}\right\|_{L^{p}(\O)}
    +\left\| \frac{(1-\eta_2)\u}{|x|}\right\|_{L^{p}(\O)}.
\end{equation}
where $\eta_2(x)$ satisfies
$$0\leq \eta_2(x)\leq 1,\q \eta_2(x)=1 \q \text{in } \O_{\epsilon/2},
\q \supp\eta_2 \subset \bar{\O}_{\epsilon},
\q |D\eta_2|\leq \frac{4}{\epsilon}.$$

We extend the vector $\hat{\u}$ on $\tilde{\O}_{\epsilon}$ to the lower half space
\begin{equation}\label{3.15}
\tilde{\u}(y)=\begin{cases} (\hat{u}_1(y_1,y_2,y_3),
\hat{u}_2(y_1,y_2,y_3), \hat{u}_3(y_1,y_2,y_3))
 \q& \text{ if } y\in \tilde{\O}_{\epsilon}\subset\Bbb R_{+}^3, \\
 (-\hat{u}_1(y_1,y_2,-y_3), -\hat{u}_2(y_1,y_2,-y_3),
 \hat{u}_3(y_1,y_2,-y_3))& \text{ if } y
 \in -\tilde{\O}_{\epsilon}\subset \Bbb R_{-}^3
\end{cases}
\end{equation}
and take the even extension for $\hat{\eta}_2(y)=\eta_2(x)$
 on $\tilde{\O}_{\epsilon}$
$$\tilde{\eta}_2(y)=\begin{cases} \hat{\eta}_2(y_1,y_2,y_3)
 \q& \text{ if } y\in \tilde{\O}_{\epsilon}\subset\Bbb R_{+}^3, \\
 \hat{\eta}_2(y_1,y_2,-y_3)\q& \text{ if } y
 \in -\tilde{\O}_{\epsilon}\subset \Bbb R_{-}^3.
\end{cases}
$$
Since $\tilde{\eta}_2\tilde{\u}\in W_0^{1,p}(\mathbb R^3)$, the classical Hardy inequality
gives that
\begin{equation}\label{2.25}
\left\|\frac{\tilde{\eta}_2\tilde{\u}}{|y|}\right\|_{L^{p}(\mathbb R^3)}\leq C\left\|D(\tilde{\eta}_2\tilde{\u})\right\|_{L^p(\mathbb{R}^3)}
\leq C\left(\left\|\dv_y(\tilde{\eta}_2\tilde{\u})\right\|_{L^p(\mathbb{R}^3)}
+\|\curl_y(\tilde{\eta}_2\tilde{\u})\|_{L^{p}(\mathbb{R}^3)}\right).
\end{equation}
The extensions of the vector $\hat{\u}$ and
the function $\hat{\eta}_2$ imply that
\begin{equation}\label{2.26}
\left\|\dv_y(\tilde{\eta}_2\tilde{\u})\right\|_{L^p(\mathbb{R}^3)}
\leq 2\left(\left\|\dv_y\hat{\u}\right\|_{L^p(\tilde{\O}_{\epsilon})}
+\left\|(\nabla\hat{\eta}_2) \hat{\u}\right\|_{L^p(\tilde{\O}_{\epsilon})}\right).
\end{equation}
The representation of the div operator under the
curvilinear coordinate system shows that
\begin{equation}\label{2.27}
\left\|\dv_y\hat{\u}\right\|_{L^p(\tilde{\O}_{\epsilon})}
\leq 4 \epsilon \|\nabla \u\|_{L^p(\O_{\epsilon})}+C(p, \O)\left(\left\|\dv_x\u\right\|_{L^p(\O_{\epsilon})}
+\|\u\|_{L^p(\O_{\epsilon})}\right).
\end{equation}
The last term of \eqref{2.26} can be estimated by
$$\left\|(\nabla\hat{\eta}_2) \hat{\u}\right\|_{L^p(\tilde{\O}_{\epsilon})}
\leq C(\epsilon,p,\O) \left\|\u\right\|_{L^p(\tilde{\O}_{\epsilon})}.
$$
Combing the above three inequalities, we now thus have
$$
\aligned
\left\|\dv_y(\tilde{\eta}_2\tilde{\u})\right\|_{L^p(\mathbb{R}^3)}
&\leq  4 \epsilon \|\nabla \u\|_{L^p(\O)}
+C(p, \epsilon, \O)\left(\left\|\dv_x\u\right\|_{L^p(\O)}
+\|\u\|_{L^p(\O)}\right)\\
&\leq C(p, \O)\left(\left\|\dv\u\right\|_{L^p(\O)}
+\|\curl\u\|_{L^p(\O)}\right),
\endaligned
$$
where we have chosen $\epsilon$ being small
and let $\epsilon $ be fixed.
Similarly, we can get the estimate for the curl part
$$
\left\|\curl_y(\tilde{\eta}_2\tilde{\u})\right\|_{L^p(\mathbb{R}^3)}
\leq C(p, \O)\left(\left\|\dv\u\right\|_{L^p(\O)}
+\|\curl\u\|_{L^p(\O)}\right).
$$
By the differmorphism mapping \eqref{6.12},
then using the estimates on the div part and the curl part we have
$$
\left\| \frac{\eta_2\u}{|x|}\right\|_{L^{p}(\O)}
\leq C(p)\left\|\frac{\tilde{\eta}_2\tilde{\u}}{|y|}\right\|_{L^{p}(\mathbb R^3)}
\leq C(p, \O)\left(\left\|\dv\u\right\|_{L^p(\O)}
+\|\curl\u\|_{L^p(\O)}\right).
$$
At last, combing the estimate on the last term of \eqref{2.25}
$$\left\| \frac{(1-\eta_2)\u}{|x|}\right\|_{L^{p}(\O)}
\leq  C(\epsilon,p,\O) \left\|\u\right\|_{L^p(\O)}
\leq C(p, \O)\left(\left\|\dv\u\right\|_{L^p(\O)}
+\|\curl\u\|_{L^p(\O)}\right),
$$
the inequality \eqref{6.22} follows and the proof is completed.
\end{proof}

\begin{Rem}\label{Rem2.5}
The conclusion of Theorem \ref{thm3.2} is still true for the vector fields
with the normal components on the boundary vanishing under the assumption
that the domain is simply-connected. The difference of the proof
is replaced the extension \eqref{3.15} by
$$
\tilde{\u}(y)=\begin{cases} (\hat{u}_1(y_1,y_2,y_3),
\hat{u}_2(y_1,y_2,y_3), \hat{u}_3(y_1,y_2,y_3))
 \q& \text{ if } y\in \tilde{\O}_{\epsilon}\subset\Bbb R_{+}^3, \\
 (\hat{u}_1(y_1,y_2,-y_3), \hat{u}_2(y_1,y_2,-y_3),
 -\hat{u}_3(y_1,y_2,-y_3))& \text{ if } y
 \in -\tilde{\O}_{\epsilon}\subset \Bbb R_{-}^3.
\end{cases}
$$
\end{Rem}

\subsection*{Acknowledgments.}
The research work was partly supported by the
National Natural Science Foundation of China grant no. 11171111.

\end {document}